\documentclass[11pt]{article}
\usepackage{amsmath, amssymb, amsthm, fullpage}
\date{}
\title{\vspace{-1cm} Maximizing the number of independent sets of a fixed
size}
\author{
Wenying Gan \thanks{Department of Mathematics, ETH,  8092 Zurich, Switzerland. Email: ganw@math.ethz.ch.}
\and
Po-Shen Loh \thanks{Department of Mathematical Sciences, Carnegie Mellon
University, Pittsburgh, PA 15213. E-mail: ploh@cmu.edu. Research supported
in part by NSF grant DMS-1201380 and by a USA-Israel BSF Grant.}
\and
Benny Sudakov \thanks{Department of Mathematics, ETH, 8092 Zurich, Switzerland and Department of Mathematics, UCLA, Los Angeles, CA 90095.
Email: bsudakov@math.ucla.edu. Research supported in part by SNSF grant 200021-149111 and by a USA-Israel BSF grant.}
}

\theoremstyle{plain}
\newtheorem{THM}{Theorem}[section]
\newtheorem*{THM*}{Theorem}
\newtheorem{PROP}[THM]{Proposition}
\newtheorem{LEMMA}[THM]{Lemma}

\newtheorem{CONJ}[THM]{Conjecture}

\theoremstyle{definition}

\begin{document}
\maketitle

\begin{abstract}
  Let $i_t(G)$ be the number of independent sets of size $t$ in a graph
  $G$.  Engbers and Galvin asked how large $i_t(G)$ could be in graphs with
  minimum degree at least $\delta$. They further conjectured that when $n\geq
  2\delta$ and $t\geq 3$, $i_t(G)$ is maximized by the complete bipartite
  graph $K_{\delta, n-\delta}$. This conjecture has drawn the attention of
  many researchers recently.  In this short note, we prove this conjecture.
\end{abstract}

\section{Introduction} \label{sec:introduction}

Given a finite graph $G$, let $i_t(G)$ be the number of independent sets of
size $t$ in a graph, and let $i(G)=\sum_{t\geq 0} i_t(G)$ be the total
number of independent sets. There are many extremal results on $i(G)$ and
$i_t(G)$ over families of graphs with various degree restrictions.  Kahn
\cite{K01} and Zhao \cite{Z10} studied the maximum number of independent
sets in a $d$-regular graph.  Relaxing the regularity constraint to a
minimum degree condition, Galvin \cite{G11}  conjectured that the number of
independent sets in an $n$-vertex graph with minimum degree $\delta\leq
\frac{n}{2}$ is maximized by a complete bipartite graph $K_{\delta,
n-\delta}$. This conjecture was recently proved (in stronger form) by
Cutler and Radcliffe \cite{CR13} for all $n$ and $\delta$, and they
characterized the extremal graphs for $\delta > \frac{n}{2}$ as well.

One can further strengthen Galvin's conjecture by asking whether the
extremal graphs also simultaneously maximize the number of independent sets of size $t$, for all $t$.
This claim unfortunately is too strong, as there are easy counterexamples
for $t = 2$. On the other hand, no such examples are known for $t\geq 3$.
Moreover, in this case Engbers and Galvin \cite{EG12} made the following
conjecture.
\begin{CONJ}
  \label{conj:main}
  For every $t \geq 3$ and $\delta \leq n/2$, the complete bipartite graph
  $K_{\delta, n-\delta}$ maximizes the number of independent sets of size
  $t$, over all $n$-vertex graphs with minimum degree at least $\delta$.
\end{CONJ}

Engbers and Galvin \cite{EG12} proved this for $\delta = 2$ and $\delta =
3$, and for all $\delta > 3$, they proved it when $t \geq 2\delta+1$.
Alexander, Cutler, and Mink \cite{ACM12} proved it for the entire range of
$t$ for bipartite graphs, but it appeared nontrivial to extend the result
to general graphs.  The first result for all graphs and all $t$ was
obtained by Law and McDiarmid \cite{LM12}, who proved the statement for
$\delta\leq n^{1/3}/2$.  This was improved by Alexander and Mink
\cite{AM13}, who required that $\frac{(\delta+1)(\delta+2)}{3} \leq n$. In
this short note, we completely resolve this conjecture.

\begin{THM}\label{thm:main}
  Let $\delta \leq n/2$.  For every $t \geq 3$, every $n$-vertex graph $G$
  with minimum degree at least $\delta$ satisfies $i_t(G)\leq
  i_t(K_{\delta, n-\delta})$, and when $t \leq \delta$, $K_{\delta,
  n-\delta}$ is the unique extremal graph.
\end{THM}

\section{Proof}

We will work with the complementary graph, and count cliques instead of
independent sets.  Cutler and Radcliffe \cite{CR13} also discovered that
the complement was more naturally amenable to extension; we will touch on
this in our concluding remarks.  Let us define some notation for use in our
proof.  A $t$-clique is a clique with $t$ vertices.  For a graph $G = (V,
E)$, $\overline{G}$ is its complement, and $k_t(G)$ is the number of
$t$-cliques in $G$.  For any vertex $v\in V$, $N(v)$ is the set of the
neighbors of $v$, $d(v)$ is the degree of $v$, and $k_t(v)$ is the number
of $t$-cliques which contain vertex $v$.  Note that $\sum_{v\in V} k_t(v)=
tk_t(G)$.  We also define $G+H$ as the graph consisting of the disjoint
union of two graphs $G$ and $H$.  By considering the complementary graph,
it is clear that our main theorem is equivalent to the following statement.

\begin{PROP}
  \label{prop:cmp}
  Let $1 \leq b \leq \Delta+1$.  For all $t \geq 3$, $k_t(G)$ is maximized
  by $K_{\Delta+1}+K_{b}$, over $(\Delta+1 + b)$-vertex graphs with maximum
  degree at most $\Delta$.  When $t \leq b$, this is the unique extremal
  graph, and when $b < t \leq \Delta+1$, the extremal graphs are
  $K_{\Delta+1} + H$, where $H$ is an arbitrary $b$-vertex graph.
\end{PROP}

\noindent \textbf{Remark.} When $b \leq 0$, the number of $t$-cliques in
graphs with maximum degree at most $\Delta$ is trivially maximized by the
complete graph.  On the other hand, when $b > (\Delta+1)$, the problem
becomes much more difficult, and our investigation is still ongoing.  This
paper focuses on the first complete segment $1 \leq b \leq \Delta+1$,
which, as mentioned in the introduction, was previously attempted in
\cite{AM13, EG12, LM12}.

\medskip

Although our result holds for all $t \geq 3$, it turns out that the main
step is to establish it for the case $t=3$ using induction and
double-counting.  Afterward, a separate argument will reduce the general $t
> 3$ case to this case of $t=3$.

\begin{LEMMA}
  Proposition \ref{prop:cmp} is true when $t = 3$.
  \label{lem:t=3}
\end{LEMMA}

\begin{proof}
  We proceed by induction on $b$.  The base case $b=0$ is trivial.  Now
  assume it is true for $b-1$.  Suppose first that $k_3(v) \leq
  \binom{b-1}{2}$ for some vertex $v$.  Applying the inductive hypothesis
  to $G - v$, we see that
  \begin{displaymath}
    k_3(G)
    \leq
    k_3(G-v) + k_3(v)
    \leq
    \binom{\Delta+1}{3} + \binom{b-1}{3} + \binom{b-1}{2}
    \leq
    \binom{\Delta+1}{3} + \binom{b}{3} \,,
  \end{displaymath}
  and equality holds if and only if $G-v$ is optimal and $k_3(v) =
  \binom{b-1}{2}$.  By the inductive hypothesis, $G-v$ is $K_{\Delta+1} +
  H'$, where $H'$ is a $(b-1)$-vertex graph.  The maximum degree
  restriction forces $v$'s neighbors to be entirely in $H'$, and so $G =
  K_{\Delta+1} + H$ for some $b$-vertex graph $H$. Moreover, since $k_3(v) =
  \binom{b-1}{2}$ we get that for $b\geq 3$, $H$ is a clique.

  This leaves us with the case where $k_3(v) > \binom{b-1}{2}$ for every
  vertex $v$, which forces $b \leq d(v) \leq \Delta$.  We will show that
  here, the number of 3-cliques is strictly suboptimal.  The number of
  triples $(u, v, w)$ where $uv$ is an edge and $vw$ is not an edge is
  clearly $\sum_{i=1}^n d(v)(n-1-d(v))$.  Also, every set of 3 vertices
  either contributes 0 to this sum (if either all or none of the 3 edges
  between them are present), or contributes 2 (if they induce exactly 1 or
  exactly 2 edges).  Therefore,
  \begin{displaymath}
    2 \left[ \binom{n}{3} - (k_3(G)+k_3(\overline{G})) \right]
    =
    \sum_{v \in V} d(v)(n-1-d(v)) \,.
  \end{displaymath}
  Rearranging this equality and applying $k_3(\overline{G})\geq 0$, we find
  \begin{equation}
    \label{eqn:k3i}
    k_3(G)\leq \binom{n}{3}-\frac{1}{2}\sum_{{v\in V}}d(v)(n-1-d(v)) \,.
  \end{equation}
  Since we already bounded $b\leq d(v)\leq \Delta$, and $b + \Delta = n-1$
  by definition, we have $d(v)(n-1-d(v))\geq b\Delta$. Plugging this back
  into \eqref{eqn:k3i} and using $n = (\Delta + 1) + b$,
  \begin{displaymath}
    k_3(G)\leq \binom{n}{3}-\frac{nb \Delta}{2}
    =
    \binom{\Delta+1}{3}+\binom{b}{3}-\frac{b(\Delta+1-b)}{2}
    < \binom{\Delta+1}{3}+\binom{b}{3} \,,
  \end{displaymath}
  because $b \leq \Delta$.  This completes the case where every vertex has
  $k_3(v) > \binom{b-1}{2}$.
\end{proof}

We reduce the general case to the case of $t=3$ via the
following variant of the celebrated theorem of Kruskal-Katona \cite{K68, K63},
which appears as Exercise 31b in Chapter 13 of from Lov\'asz's book
\cite{Lo07}.  Here, the generalized binomial coefficient $\binom{x}{k}$ is
defined to be the product $\frac{1}{k!} (x)(x-1)(x-2) \cdots (x-k+1)$,
which exists for non-integral $x$.

\begin{THM}\label{thm:kk}
Let $k \geq 3$ be an integer, and let $x \geq
  k$ be a real number.  Then, every graph with exactly $\binom{x}{2}$ edges
  contains at most $\binom{x}{k}$ cliques of order $k$.
\end{THM}

We now use Lemma \ref{lem:t=3} and Theorem \ref{thm:kk} to finish the
general case of Proposition \ref{prop:cmp}.

\begin{LEMMA}
  \label{lemma:t>3}
  If Proposition \ref{prop:cmp} is true for $t = 3$, then it is also true
  for $t > 3$.
\end{LEMMA}

\begin{proof}
  Fix any $t \geq 4$.  We proceed by induction on $b$. The base case $b=0$
  is trivial. For the inductive step, assume the result is true for $b-1$.
  If there is a vertex $v$ such that $k_3(v) \leq \binom{b-1}{2}$, then by
  applying Theorem \ref{thm:kk} to the subgraph induced by $N(v)$, we find
  that there are at most $\binom{b-1}{t-1}$ cliques of order $t-1$ entirely
  contained in $N(v)$.  The $t$-cliques which contain $v$ correspond
  bijectively to the $(t-1)$-cliques in $N(v)$, and so $k_t(v) \leq
  \binom{b-1}{t-1}$.  The same argument used at the beginning of Lemma
  \ref{lem:t=3} then correctly establishes the bound and characterizes the
  extremal graphs.

  If some $k_3(v) = \binom{\Delta}{2}$, then the maximum degree condition
  implies that the graph contains a $K_{\Delta+1}$ which is disconnected
  from the remaining $b \leq \Delta+1$ vertices, and the result also easily
  follows.  Therefore, it remains to consider the case where all
  $\binom{b-1}{2} < k_3(v) < \binom{\Delta}{2}$, in which we will prove
  that the number of $t$-cliques is strictly suboptimal.  It is well-known
  and standard that for each fixed $k$, the binomial coefficient
  $\binom{x}{k}$ is strictly convex and increasing in the real variable $x$
  on the interval $x \geq k-1$.  Hence, $\binom{k}{k} = 1$ implies that
  $\binom{x}{k} < 1$ for all $k-1 < x < k$, and so Theorem \ref{thm:kk}
  then actually applies for all $x \geq k-1$.  Thus, if we define $u(x)$ to
  be the positive root of $\binom{u}{2} = x$, i.e., $u(x) = \frac{1 +
    \sqrt{1+8x}}{2}$, and let
  \begin{equation}
    f_t(x)
    =
    \left\{
      \begin{array}{ll}
        {0} & \text{if } u(x) < t-2 \\
        {\binom{u(x)}{t-1}} & \text{if } u(x) \geq t-2 \,,
      \end{array}
      \right.
  \end{equation}
  the application of Kruskal-Katona in the previous paragraph establishes
  that $k_t(v) \leq f_t(k_3(v))$.

  We will also need that $f_t(x)$ is strictly convex for $x >
  \binom{t-2}{2}$.  For this, observe that by the generalized product rule,
  $f_t'(x) = u' \cdot \left[ (u-1)(u-2)\cdots(u-(t-2)) + \cdots +
  u(u-1)\cdots(u-(t-3)) \right]$, which is $u'(x)$ multiplied by a sum of
  $t-1$ products.  Since $u'(x) = \frac{2}{\sqrt{1+8x}}$, for any
  constant $C$, $(u')(u-C) = 1 - \frac{2C-1}{\sqrt{1+8x}}$.  Note that
  this is a positive increasing function when $C \in \{1, 2\}$ and $x
  > \binom{t-2}{2}$.  In particular, since $t \geq 4$, each of the
  $t-1$ products contains a factor of $(u-1)$ or $(u-2)$, or possibly
  both; we can then always select one of them to absorb the $(u')$
  factor, and conclude that $f_t'(x)$ is the sum of $t-1$ products, each
  of which is composed of $t-2$ factors that are positive increasing
  functions on $x > \binom{t-2}{2}$.  Thus $f_t(x)$ is strictly convex
  on that domain, and since $f_t(x) = 0$ for $x \leq \binom{t-2}{2}$, it
  is convex everywhere.

  If $t = \Delta+1$, there will be no $t$-cliques in $G$ unless $G$
  contains a $K_{\Delta+1}$, which must be isolated because of the maximum
  degree condition; we are then finished as before.  Hence we may assume $t
  \leq \Delta$ for the remainder, which in particular implies that $f_t(x)$
  is strictly convex and strictly increasing in the neighborhood of $x
  \approx \binom{\Delta}{2}$.  Let the vertices be $v_1, \ldots, v_n$, and
  define $x_i = k_3(v_i)$.  We have $tk_t(G)=\sum_{v\in V} k_t(v) \leq
  \sum_{i=1}^n {f_t(x_i)}$, and so it suffices to show that $\sum f_t(x_i)
  < t \binom{\Delta+1}{t} + t \binom{b}{t}$ under the following conditions,
  the latter of which comes from Lemma \ref{lem:t=3}.
  \begin{equation}
    \binom{b-1}{2} < x_i < \binom{\Delta}{2} \,;
    \quad\quad
    \sum_{i=1}^n x_i \leq 3 \binom{\Delta+1}{3} + 3 \binom{b}{3} \,.
    \label{eq:domain-x}
  \end{equation}

  To this end, consider a tuple of real numbers $(x_1, \ldots, x_n)$ which
  satisfies the conditions.  Although \eqref{eq:domain-x} constrains each
  $x_i$ within an open interval, we will perturb the $x_i$ within the
  closed interval which includes the endpoints, in such a way that the
  objective $\sum f_t(x_i)$ is nondecreasing, and we will reach a tuple
  which achieves an objective value of exactly $t \binom{\Delta+1}{t} + t
  \binom{b}{t}$.  Finally, we will use our observation of strict convexity
  and monotonicity around $x \approx \binom{\Delta}{2}$ to show that one of
  the steps strictly increased $\sum f_t(x_i)$, which will complete the
  proof.

  First, since the upper limit for $\sum x_i$ in \eqref{eq:domain-x} is
  achievable by setting $\Delta+1$ of the $x_i$ to $\binom{\Delta}{2}$ and
  $b$ of the $x_i$ to $\binom{b-1}{2}$, and $f_t(x)$ is nondecreasing, we
  may replace the $x_i$'s with another tuple which has equality for $\sum
  x_i$ in \eqref{eq:domain-x}, and all $\binom{b-1}{2} \leq x_i \leq
  \binom{\Delta}{2}$.  Next, by convexity of $f_t(x)$, we may push apart
  $x_i$ and $x_j$ while conserving their sum, and the objective is
  nondecreasing.  After a finite number of steps, we arrive at a tuple in
  which all but at most one of the $x_i$ is equal to either the lower limit
  $\binom{b-1}{2}$ or the upper limit $\binom{\Delta}{2}$, and $\sum x_i =
  3 \binom{\Delta+1}{3} + 3 \binom{b}{3}$.  However, since this value of
  $\sum x_i$ is achievable by $\Delta+1$ many $\binom{\Delta}{2}$'s and $b$
  many $\binom{b-1}{2}$'s, this implies that in fact, the tuple of $x_i$'s
  has precisely this form.  (To see this, note that by an affine
  transformation, the statement is equivalent to the fact that if $n$ and
  $k$ are integers, and $0 \leq y_i \leq 1$ are $n$ real numbers which sum
  to $k$, all but one of which is at an endpoint, then exactly $k$ of the
  $y_i$ are equal to 1 and the rest are equal to 0.)  Thus, our final
  objective is equal to
  \begin{displaymath}
    (\Delta+1) \binom{\Delta}{t-1} + b \binom{b-1}{t-1}
    =
    t \binom{\Delta+1}{t} + t \binom{b}{t} \,,
  \end{displaymath}
  as claimed.  Finally, since some $x_i$ take the value
  $\binom{\Delta}{2}$, the strictness of $f_t(x)$'s monotonicity and
  convexity in the neighborhood $x \approx \binom{\Delta}{2}$ implies that
  at some stage of our process, we strictly increased the objective.
  Therefore, in this case where all $\binom{b-1}{2} < k_3(v) <
  \binom{\Delta}{2}$, the number of $t$-cliques is indeed sub-optimal, and
  our proof is complete.
\end{proof}

\section{Concluding remarks}
\label{sec:conclusion}

The natural generalization of Proposition \ref{prop:cmp} considers the
maximum number of $t$-cliques in graphs with maximum degree $\Delta$ and
$n = a(\Delta+1) + b$ vertices, where $0 \leq b < \Delta + 1$.  In the
language of independent sets, this question was also proposed
by Engbers and Galvin \cite{EG12}.  The case $a = 0$ is trivial, and
Proposition \ref{prop:cmp} completely solves the case $a = 1$.  We believe that
also for $a>1$ and $t\geq 3$, $k_t(G)$ is maximized by $a K_{\Delta+1}+K_{b}$, over
$(a(\Delta+1) + b)$-vertex graphs with maximum degree at most $\Delta$.

An easy double-counting argument shows that it is true when $b=0$.  When $a
\geq 2$ and $b > 0$, the problem seems considerably more delicate.
Nevertheless, the same proof that we used in Lemma \ref{lemma:t>3} (\emph{mutatis
mutandis}) shows that the general case $t>3$ of this problem can be reduced
to the case $t = 3$.  Therefore, the most intriguing and challenging part
is to show that $a K_{\Delta+1}+K_{b}$ maximizes the number of triangles
over all graphs with $(a(\Delta+1) + b)$ vertices and maximum degree at
most $\Delta$.  We have some partial results on this main case, but our
investigation is still ongoing.

\end{document}